\documentclass[twoside,12pt,leqno]{amsart}
\usepackage{pstricks,graphics,epsf,graphicx,amssymb,latexsym,verbatim,enumerate}
\usepackage{hyperref}
\hypersetup{citecolor=[rgb]{1,0,0,}, linkcolor=[rgb]{0,0,1}, colorlinks=true}

\usepackage[T2A,T1]{fontenc}
\usepackage[utf8]{inputenc}
\usepackage[russian,french,english]{babel}

\theoremstyle{plain}
\newtheorem{theorem}{Theorem}
\newtheorem{lemma}[theorem]{Lemma}

\theoremstyle{definition}

\theoremstyle{remark}

\begin{document}

\newcommand{\diag}{\textup{diag}}
\newcommand{\F}{\mathbb{F}}
\newcommand\GL{\textup{GL}}
\newcommand\n{\newline}
\newcommand\w{\,\textup{wr}\,}
\newcommand\vstr{\phantom{\kern-20pt$\sum_{i_j}^{k^l}$}}
\newcommand{\Z}{\mathbb{Z}}
\renewcommand{\ge}{\geqslant}
\renewcommand{\le}{\leqslant}
\renewcommand{\geq}{\geqslant}
\renewcommand{\leq}{\leqslant}

\hyphenation{Hilfssatz uni-potent}

\title[\tiny\upshape\rmfamily The shape of solvable groups with odd order]{}


\begin{center}\large\sffamily\mdseries
The shape of solvable groups with odd order
\end{center}

\author{{\sffamily S.\,P. Glasby}}

\begin{abstract}
The minimal composition length, \kern-1.4pt$c$, of a solvable
group with solvable length $d$ satisfies $9^{(d-3)/9}< c< 9^{(d+1)/5}$.
The minimal composition length, $c^o$, of a group with odd order
and solvable length $d$ satisfies $7^{(d-2)/5}< c^o< 2^d$.
\end{abstract}

\maketitle
\centerline{\noindent 2000 Mathematics subject classification:
20F16, 20F14, 20E34}

\section{Introduction}
Let $c(G)$ and $d(G)$ denote the composition length and
solvable (or derived length) of a finite and solvable $G$.
All groups in this paper are finite and solvable unless otherwise
stated. If $|G|=p_1\cdots p_r$ where the $p_i$ are primes, then
$G$ has a composition factor of order $p_i$ for $i=1,\dots,r$ and
$c(G)=r$. The derived series for $G$ is defined recursively:
$G^{(0)}:=G$ and $G^{(i+1)}:=[G^{(i)},G^{(i)}]$ for $i\ge0$. By
definition $d(G)$ is the smallest $d\ge0$ such that $G^{(d)}=1$.

We seek to understand solvable and nilpotent  groups that have solvable
length $d$, and smallest possible composition length. Set
\begin{align*}
 c_N(d)&:=\min\{c(G)\mid G\text{ is nilpotent and }d(G)=d\}\\
 c_S(d)&:=\min\{c(G)\mid G\text{ is solvable and }d(G)=d\}.
\end{align*}
Let $c^o_N(d)$ (resp. $c^o_S(d)$) be defined similarly except that $G$
ranges over nilpotent (resp. solvable) groups of {\em odd} order.
Section \ref{Solvable} is devoted to a proof of the result in the
abstract.

A simple way to construct groups with large solvable length is via
permutational wreath products. Let $S_m$ denote the symmetric group of
degree $m$. The group $S_m\w S_n$ can be viewed as an imprimitive
subgroup of $S_{mn}$. (We shall not view $S_m\w S_n$ as a subgroup of
$S_{m^n}$ with product action.) If $H\le S_m$ and $K\le S_n$ are both
transitive, then $H\w K\le S_{mn}$ is transitive and $d(H\w
K)=d(H)+d(K)$, see \cite[Corollary~1]{mN72}. The wreath product
$G_r=S_2\w\cdots\w S_2$ with $r$
copies of $S_2$ has $d(G_r)=r$ and $c(G_r)=2^r-1$. P. Hall
\cite[Satz III.2.12]{H67} showed that $G^{(i)}\le\gamma_{2^i}(G)$.
If $G$ is a $p$-group and $d(G)=d>1$, then $2^{d-1}+1\le c(G)$ holds as
$2^{d-1}\le c(G/\gamma_{2^{d-1}}(G))\le c(G/G^{(d-1)})$. The following
bounds hold for $d\ge1$:
\[
 2^{d-1}\le c_N(d)\le 2^d-1.
\]
Therefore $d=\lfloor\log_2 c_N(d)\rfloor+1$. In Section~\ref{Examples}
some general remarks concerning the sharpening bounds for $c_N(d)$
and $c_S(d)$ are made, and the difficulty of constructing
groups with solvable length $d$ and minimal composition length is
considered.

\section{Solvable groups}\label{Solvable}

Given a solvable group with solvable length $d$ and minimal
composition length $c$, we bound $d$ in terms of $c$,
and $c$ in terms of $d$.

\begin{theorem}\label{Main}
Denote by $c$ (resp. $c_0$) the minimal composition length
of a solvable group (resp. odd-order group) with solvable
length~$d\kern-1pt>\kern-2.4pt0$.~Then
\begin{itemize}
\item[(a)] $\gamma\log_2 c-\frac23< d< (\gamma+1)\log_2 c+3$ where
$\gamma=5\log_9 2\approx 1.58$,
\item[(b)] $\log_2 c_0< d< (\gamma_0+1)\log_2 c_0+2$ where
$\gamma_0=2\log_7 2\approx 0.71$.
\end{itemize}
Hence $9^{(d-3)/9}< c_S(d)< 9^{(d+1)/5}$ and
$7^{(d-2)/5}< c^o_S(d)< 2^d$ where $9^{1/9}\approx1.27\cdots$,
$9^{1/5}\approx1.55\cdots$, and $7^{1/5}\approx1.47\cdots$.
\end{theorem}

\begin{proof}
(a) Suppose that $G$ is a group with solvable length $d>0$,
and minimal composition length. Then $c(G)$ equals $c:=c_S(d)$. If $N$
is a nontrivial normal subgroup of $G$, then $c(G/N)<c(G)$ and so
$d(G/N)<d(G)$. It follows that $G$ cannot have distinct minimal normal
subgroups $N_1,N_2$ otherwise $G$ embeds in $G/N_1\times G/N_2$ and
$d(G)>\max\{d(G/N_1),d(G/N_2)\}$. Let $N$ be {\em the} unique minimal normal
subgroup of $G$. As $N$ is characteristically simple, it is an
elementary abelian $p$-group for some prime $p$. Let $P$ be the
maximal normal $p$-subgroup of $G$. Then $P\ne1$, and $G$ has no normal
subgroups with order coprime to $p$. By a theorem of Hall and Higman
\cite[Hilfssatz VI.6.5]{H67} $G/P$ acts faithfully and completely
reducibly on the vector space $P/\Phi(P)$. We shall view $P/\Phi(P)$
as an $r$-dimensional
vector space over the field $\F_p$ with $p$ elements. As $P/\Phi(P)$ is
abelian $d(G)\le d(G/P)+1+d(\Phi(P))$ holds. Suppose $|\Phi(P)|=p^s$.
Then $s:=c(\Phi(P))$ and $d(\Phi(P))\le\log_2s+1$.
Because $G/P\le\GL(r,\F_p)$ is completely reducible,
\cite[Theorem C]{mN72} gives
\[
 d(G/P)\le5\log_9(r/8)+8=\gamma\log_2(r/8)+8
\]
where $\gamma=5\log_9 2$. Since
\[
  c(G)=c(G/P)+c(P/\Phi(P))+c(\Phi(P)),
\]
we see that $c=c(G/P)+r+s$ and $s\le c-r$. Therefore
\begin{align*}
 d=d(G)&\le (\gamma\log_2(r/8)+8)+1+(\log_2s+1)\\
       &\le\gamma\log_2(r/8)+\log_2(c-r)+10
\end{align*}
Using calculus, the maximum of $\gamma\log_2(r/8)+\log_2(c-r)+10$,
with $0<r<c$, occurs when $r=\gamma c/(\gamma+1)$. Thus
$d\le(\gamma+1)\log_2 c+\delta$ where
\[
 \delta=\gamma\log_2(\gamma/8)-(\gamma+1)\log_2(\gamma+1)+10\approx2.78<3.
\]
This establishes the upper bound for (a).

The lower bound for (a) is obtained by constructing five families of
transitive permutation groups $G_n\le S_n$. Our exposition is
influenced by \cite{mN72}. Let $G_9$ denote the maximal solvable
primitive permutation subgroup $\GL(2,3)\ltimes 3^2\le S_9$, and set
$G_8=\GL(2,3)\le S_8$. Let
$m=9^r$, and define $G_m\le S_m$ to be $G_9\w\cdots\w G_9$ with $r$
copies of $G_9$. Set $G_{2m}=S_2\w G_m$, $G_{3m}=S_3\w G_m$,
$G_{4m}=S_4\w G_m$, and $G_{8m}=G_8\w G_m$. Note that $d(G_m)=5r$
and $c(G_m)=7(m-1)/8$ because $|G_m|=|G_9|^{(m-1)/8}$. Thus
\[
  c_S(5r)\le 7(m-1)/8< 7m/8,\quad\textup{ and}\quad
    5\log_9 c_S(5r)-2/3< 5\log_9(7m/8)-2/3<5r
\]
because $\log_9 m=r$ and $5\log_9(7/8)-2/3\approx-0.97$. Similar
calculations are
summarized below. In each case $c(G_n)$ equals
$(k_nm-7)/8$ for some $k_n\in\Z$, and we abbreviate
$5\log_9(k_n/8)-2/3$ by $x_n$.

\begin{center}
\begin{tabular}{|c|c|c|c|c|c|} \hline
$G_n$&$d(G_n)$&$|G_n|$&$c(G_n)$&$x_n$&$\log_9 c(G_n)-\frac23$\vstr\\
\hline
$G_m$&$5r$&$|G_9|^{(m-1)/8}$&$\hfill(7m-7)/8$&$\approx-0.97$&$<5r$\\
$G_{2m}$&$5r+1$&$|S_2|^m|G_m|$&$(15m-7)/8$&$\approx0.8$&$<5r+1$\\
$G_{3m}$&$5r+2$&$|S_3|^m|G_m|$&$(23m-7)/8$&$\approx1.7$&$<5r+2$\\
$G_{4m}$&$5r+3$&$|S_4|^m|G_m|$&$(39m-7)/8$&$\approx2.9$&$<5r+3$\\
$G_{8m}$&$5r+4$&$|G_8|^m|G_m|$&$(47m-7)/8$&$\approx3.4$&$<5r+4$\\
\hline
\end{tabular}
\end{center}
In all five cases $5\log_9 c(G_n)-2/3<d(G_n)$ holds. Indeed,
if $\lceil x\rceil$ denotes the least integer $\ge x$, then
$\lceil5\log_9 c(G_n)-2/3\rceil$ equals $d(G_n)$.
Since $c_S(d(G_n))\le c(G_n)$, the lower bound in part (a)
holds for $d\ge1$.

\vskip2mm
\noindent
(b) The proof of this part follows the same pattern as part (a). The
upper bound
is proved similarly except instead of using \cite{mN72} we use
the following result \cite[Theorem 4b]{pP}:
If $G/P\le\GL(r,\F)$ and $|G/P|$ is odd, then
$d(G/P)\le2\log_7(r/5)+3=\gamma_0\log_2(r/5)+3$
where $\gamma_0=2\log_7 2$. Thus
$d\le(\gamma_0+1)\log_2 c_0+\delta_0$ where
\[
 \delta_0=\gamma_0\log_2(\gamma_0/5)-(\gamma_0+1)\log_2(\gamma_0+1)+5
          \approx1.7<2.
\]
This establishes the upper bound for (b).

The lower bound for (b) follows from the trivial observation that
$c^o_S(d)\le c^o_N(d)$, and the fact that P.~Hall \cite[Satz
III.17.7]{H67} constructed groups of order $p^{2^d-1}$ and solvable
length $d$ for each odd prime $p$. For aesthetic reasons, we outline
an argument that mirrors that in part (a) even though it gives a poorer
bound $\gamma_0\log_2 c_0+\frac23<d$ than $\log_2 c_0<d$.
We define as in \cite{pP} transitive subgroups $H_n\le S_n$ of degree
$m=7^r$, and $3m$. Set $H_m=H_7\w\cdots\w H_7$ where
there are $r$ copies of $H_7$ and
\[
  H_7=\langle (2,3,5)(4,6,7),(1,2,3,4,5,6,7)\rangle\le S_7,
\]
and set $H_{3m}=A_3\w H_m$ where $A_3$ has order~3 and degree~3. The
bound $\gamma_0\log_2 c^0_S(d)+\frac23<d$ is a consequence of
\[
 d(H_m)=2r, d(H_{3m})=2r+1, c(H_m)=\frac{m-1}3, c(H_{3m})=\frac{4m-1}3.
\]

It follows from $\gamma\log_2 c-\frac23< d$ and
$\log_2 c_0< d$ that $c< 9^{(d+1)/5}$ and $c_0< 2^d$. Similarly,
$d< (\gamma+1)\log_2 c+3$ and $d< (\gamma_0+1)\log_2 c_0+2$ together
with
\[
\frac19<\frac{\gamma}{5(\gamma+1)}\approx 0.122\quad\text{and}\quad
\frac15<\frac{\gamma_0}{2(\gamma_0+1)}\approx 0.208,
\]
imply $9^{(d-3)/9}<c$ and $7^{(d-2)/5}<c_0$. This completes the proof.
\end{proof}

\section{Examples}\label{Examples}

The bounds for $c_N(d)$ given in the introduction can be improved.
P.~Hall \cite[Satz III.7.10]{H67} proved $2^{d-1}+{d-1}\le c_N(d)$
by proving that a $p$-group $G$ with $G^{(i+1)}\ne 1$ satisfies
$c(G^{(i)}/G^{(i+1)})\ge 2^i+1$. No such bound exists for solvable
groups. We give an example below of an infinite residually solvable
group $G$ for which $c(G^{(2i-1)}/G^{(2i)})=1$~for~$i\ge1$.

Let $Q$ denote the quaternion group of order~8, and $E$ the
extraspecial group of order~27 and exponent~3. Denote by
$Q_n$ (resp. $E_n$) the central product of $n$ copies of $Q$
(resp. $E$) amalgamating the centers. Then $|Q_n|=2^{2n+1}$ and
$|E_n|=3^{2n+1}$. When $n=0$ the groups $Q_0$ and $E_0$ are cyclic of
order 2 and 3 respectively, and when $n>0$ both $Q_n$ and $E_n$ are
extraspecial groups. In \cite[Section~7]{GH} an iterated split extension
\[
 G=Q_0\ltimes E_0\ltimes Q_1\ltimes E_1\ltimes Q_3\ltimes
 E_4\ltimes\cdots\ltimes Q_{a_n}\ltimes E_{b_n}\ltimes\cdots.
\]
is constructed where $a_0=b_0=0$ and $a_n=3^{b_{n-1}}$,
$b_n=2^{-1+a_n}$ for $n\ge1$. The orders of the derived quotients
$G^{(i)}/G^{(i+1)}$ of $G$ are:
\[
 2,3,2^2,2,3^2,3,2^6,2,3^8,3,2^{162},2,3^{2\cdot3^{81}},3,\dots.
\]
In particular, $c(G^{(2i-1)}/G^{(2i)})=1$ for $i\ge1$. If $d\le10$ and
$d\ne7$, then $G/G^{(d)}$ has minimal composition length amongst all
solvable groups of solvable length $d$, see \cite{G05}. This is
surprising given that $c(G/G^{(d)})$ grows faster than any exponential
function of $d$. It is shown in \cite{G05} that $c_S(d)$ equals
$1,2,4,5,7,8,13,15$ when $d=1,2,3,4,5,6,7,8$.

The notation $\beta_p(d)$ is used in \cite{ENS} to denote the minimal
composition length of a $p$-group with solvable length~$d$. Clearly,
$c_N(d)=\min_p\beta_p(d)$ where $p$ ranges over the prime numbers. The
following values of $\beta_p(d)$ were known at the time of Burnside:
\[
 \beta_p(1)=1,\beta_p(2)=3,\beta_p(3)=6\text{ for $p\ge5$, and }
 \beta_2(3)=\beta_3(3)=7.
\]
It is shown in \cite{E00} that $\beta_p(4)=14$ for $p\ge5$, and in
\cite{ENS} that $\beta_p(d)\le 2^d-2$ for $p\ge5$. The best known
bounds for $c_N(d)$ are presently
\[
 2^{d-1}+3d-10\le c_N(d)\le 2^d-2.
\]
The upper bound holds when $d\ge3$, and the lower bound \cite{S99b}
which holds for $p\ge 5$ improves the bound of
Mann \cite{M00} when 
$d\ge7$. Although $c_N(d)\le c^o_N(d)$, both are $\textup{O}(2^d)$.
If it were the case that $c_S(d)$ and $c^o_S(d)$ are both
$\textup{O}(k^d)$ for some constant $k$, then Theorem~\ref{Main}
shows that $7^{1/5}<k<9^{1/5}$.

Mann~\cite{M86} investigates subgroups of $S_n$ that
have maximal order. These subgroups which are wreath
(and direct) products of $S_2,S_3$ and $S_4$, do not
improve the bound $c_S(d)=\textup{O}(9^{d/5})$ of
Theorem~\ref{Main}. The iterated wreath product $H\w\cdots\w H$ of $d$
copies of $H=S_2,S_3,S_4$ gives the bounds $c_S(d)=\textup{O}(2^d)$,
$c_S(d)=\textup{O}(3^{d/2})$ and $c_S(d)=\textup{O}(4^{d/3})$.
These bounds are less sharp because
\[
 9^{1/5}<4^{1/3}<3^{1/2}<2\quad\text{as}\quad
 1.55\cdots<1.58\cdots<1.73\cdots<2.
\]

In the proof of Theorem~\ref{Main}(a), certain groups $G_n$ were used
to establish an upper bound for $c_S(d)$. It is
natural to ask whether proper subgroups $K_n$
of $G_n$ can produce sharper upper bounds which narrow the gap in
Theorem~\ref{Main}(a). The permutation representation
$S_9\to\GL(9,\F_2)$ fixes the 8-dimensional subspace
$\{(x_1,\dots,x_9)\mid x_1+\cdots+x_9=0\}$. This observation may be
used to construct a subgroup $K_{18}$ of $G_{18}=S_2\w G_9$ of index~2
satisfying $d(G_{18})=d(K_{18})$. We define
subgroups $K_{2m}\le G_{2m}$ where $m=9^r\ge 9$ via
$K_{2(9m)}=K_{2m}\w G_9$. Although
$d(K_{2m})\kern-1.2pt=\kern-1.2ptd(G_{2m})$ and
$|G_{2m}:K_{2m}|=2^{m/18}$ is large,
$c(K_{2m})=\textup{O}(m)=c(G_{2m})$ and we obtain the same bound
$c_S(d)=\textup{O}(9^{d/5})$ as before.

It is unclear whether or not the gap $2^{d-1}+3d-10\le c_N(d)\le
2^d-2$ can be closed appreciably by examples. The group $U_n$ of
$n\times n$ unipotent upper-triangular matrices over $\F_p$ has
$c(U_n)=n(n-1)/2$ and $d(U_n)=\lfloor \log_2(n-1)\rfloor+1$. Although
$U_n$ is generated by $n-1$ elements, it has subgroups with 2 or 3
generators and maximal solvable length \cite{G99}. Estimating the
order of these subgroups appears difficult.

The wreath product $S_2\w\cdots\w S_2$, with $d$ copies of the
symmetric group $S_2$, gave rise to the bound $c_N(d)\le 2^d-1$.
It is natural to ask whether this group contains proper subgroups with
derived length $d$. The answer is negative by Lemma~\ref{CS} below.
I am grateful to Csaba Schneider for showing me a proof of
this lemma.

\begin{lemma}\label{CS}
Let $G_d=S_2\w\cdots\w S_2$ with $d$ copies of
$S_2$. Then every proper subgroup $G_d$ has solvable length less than~$d$.
\end{lemma}

\begin{proof}
We use induction on $d$. The result is true when $d=1$, and when $d=2$
because proper subgroups of $S_2\w S_2$ are abelian. Assume that
$d>2$. It suffices to prove that each maximal subgroup $M$ of $G_d$
satisfies $M^{(d-1)}=1$. Write $G_d=(H_1\times H_2)\rtimes S_2$
where $H_i\cong G_{d-1}$. The result is true if $M=H_1\times
H_2$. Suppose that $M\ne H_1\times H_2$.
Set $N_i=M\cap H_i$. Since $|G_d:M|=2$, we see $|H_i:N_i|$ equals 1
or~2. The former is impossible as $M\triangleleft\, G_d$ and $M\ne
H_1\times H_2$. Thus $|H_1:N_1|=|H_2:N_2|=2$ and $N=N_1\times
N_2\triangleleft\, G_d$. Since $M/N$ is a proper subgroup of
$G/N\cong S_2\w S_2$, and $N^{(d-2)}=1$ by induction, it follows that
$M^{(d-1)}=1$.
\end{proof}

\vskip3mm
\goodbreak
\def\efont{\scriptsize\upshape\ttfamily}
{\tiny\scshape
\begin{tabbing}
\=\hspace{70mm}\=\kill
\>Department of Mathematics    \>\\
\>Central Washington University\>\\
\>WA 98926-7424, USA           \>\\
\>\efont glasbys@gmail.com\>
\end{tabbing}
}

\end{document}